\documentclass{article}
\usepackage{amsmath}
\usepackage{amsthm}
\usepackage{amssymb}
\usepackage{graphicx}
\usepackage[xcolor=pst]{pstricks}
\usepackage{graphicx, pst-plot, pst-node, pst-text, pst-tree}
\usepackage{color}
\definecolor{lightgray}{rgb}{0.8, 0.8, 0.8}
\definecolor{darkgray}{rgb}{0.7, 0.7, 0.7}
\definecolor{darkblue}{rgb}{0, 0, .4}

\newcommand{\Av}{\operatorname{Av}}
\newcommand{\C}{\mathcal{C}}
\newcommand{\X}{\mathcal{X}}

\newcommand{\A}{\mathcal{A}}
\newcommand{\B}{\mathcal{B}}
\newcommand{\I}{\mathcal{I}}
\newcommand{\D}{\mathcal{D}}

\newtheorem{theorem}{Theorem}
\newtheorem{proposition}{Proposition}[section]
\newtheorem{lemma}{Lemma}[section]
\newtheorem{corollary}{Corollary}[section]

\newcommand{\case}[2]{{\bf \noindent Case #1: The basis of #2}}
\title{Pattern classes and priority queues}
\author{Michael Albert and M. D. Atkinson\\
Department of Computer Science\\University of Otago, New Zealand}
\begin{document}
\maketitle
\begin{abstract}
When a set of permutations comprising a pattern class $\C$ is submitted as input to a priority queue the resulting output is again a pattern class $\C'$.  The basis of $\C'$ is determined for pattern classes $\C$ whose basis elements have length 3, and is finite in these cases.  An example is given of a class $\C$ with basis $\{2431\}$ for which $\C'$ is not finitely based.
\end{abstract}
\begin{center}{\em Mathematics Subject Classification: 05A05, 68P05}\end{center}

\section{Introduction}\label{intro}

The theory of permutation patterns can trace its origin back to the abstract datatype {\tt Stack}. The earliest non trivial example of a pattern class is the set of permutations that can be sorted by a stack, which is the same as the set of 231-avoiding permutations.  This class was first studied and enumerated by Knuth in Volume 1 of \cite{knuth:the-art-of-comp:1} where he also showed how to characterize and enumerate the pattern class associated with the more complex datatype {\tt Restricted input deque}.  Over a period of more than 40 years generalizations of {\tt Stack} have appeared several times in the pattern class literature (see \cite{bona:a-survey-of-sta:} for a survey).

Stacks are characterized by their removal rule: when an item is removed from a stack it is always the most recently inserted item that is removed. By contrast the datatype {\tt Priority queue} is characterized by a different removal rule: the removed item is always the entry of \emph{smallest} value.  The associated sorting problem is trivial: a priority queue can sort every permutation.
Despite this, as we shall see below, there are pattern class aspects of priority queues that are very challenging.  

The study of pattern classes is also the study of the \emph{pattern containment order}. For our purposes it will be helpful to give a somewhat abstract description of this order. Let $\alpha$ and $\beta$ be two sequences of distinct values from some linearly ordered set. We say that $\alpha$ and $\beta$ \emph{have the same pattern} or are  \emph{isomorphic}, and write $\alpha \sim \beta$, if they have the same length $n$ and, for all $1 \leq i, j \leq n$, $\alpha_i < \alpha_j$ if and only if $\beta_i < \beta_j$. It is not even essential for this definition that $\alpha$ and $\beta$ be sequences from the same set, and allowing this ambiguity it is clear that every finite sequence $\alpha$ of distinct values from some linearly ordered set is isomorphic to a unique permutation of $[n] = \{1, 2, \dots, n \}$ (with its usual ordering $1 < 2 < \dots < n$) where $n$ is the length of $\alpha$.

Slightly more generally we say that $\alpha$ is \emph{contained as a pattern} in $\beta$ if $\alpha$ is isomorphic to some subsequence of $\beta$. In this case we write $\alpha \preceq \beta$. It is clear that $\preceq$ is a partial order when restricted to permutations. If $\alpha$ is not contained as a pattern in $\beta$ then we say that $\beta$ \emph{avoids} $\alpha$.

Pattern classes are, by definition, sets of permutations that are closed downwards in the pattern containment order on permutations.  It follows that every pattern class $\X$ can be defined in terms of a set $B$  of permutations that it avoids as
\[\X=\Av(B)=\{\sigma\mid \sigma\mbox{ does not contain }\beta\mbox{ for all }\beta\in B\}\]
The unique minimal set of avoided permutations is called the \emph{basis} of the pattern class.  Much of the study of pattern classes is concerned with investigating the structure of a pattern class given its basis.

To study priority queues however we need a partial order defined on \emph{pairs} of permutations.  Let $(\sigma, \tau)$ be any pair of permutations of length $n$ and let $S$ be any subset of $[n]$.  Let $\sigma|_S, \tau|_S$ be the subsequences of $\sigma,\tau$ whose entries are the members of $S$ and let $\alpha,\beta$ be the permutations isomorphic to $\sigma|_S, \tau|_S$.  Then we write $(\alpha,\beta)\preceq(\sigma,\tau)$ and the binary relation so defined is easily seen to be a partial order.  In a sense made precise in \cite{atkinson:permuting-machi:} this order is the 3-dimensional analogue of the 2-dimensional pattern containment order.  The two orders are connected via the following result which follows directly from the definitions.

\begin{proposition}\label{closed}
Let $\B$ be any set of pairs of permutations that is closed downwards in the order on pairs and let $\C$ be any pattern class.  Then both
\[\C\B=\{\tau\mid (\sigma,\tau)\in \B\mbox{   for some }\sigma\in \C\}\]
and
\[\B\C=\{\sigma\mid (\sigma,\tau)\in \B\mbox{   for some }\tau\in \C\}\]
are pattern classes.
\end{proposition}

Following \cite{atkinson:the-permutation:} we define an \emph{allowable pair} $(\sigma,\tau)$ of permutations to be a pair such that a priority queue can generate $\tau$ as an output sequence  if presented with $\sigma$ as an input sequence.  The set $\A$ of allowable pairs is easily seen to be closed downwards in the pair order. Indeed, it is the set of pairs that do not contain either of the pairs $(12,21)$ and $(321, 132)$ \cite{atkinson:priority-queues-and:}; thus Proposition \ref{closed} applies to $\A$.  

This paper studies the relationship between the  pattern classes $\C$ and $\C\A$ for various pattern classes $\C$.  In other words we study the pattern class generated as a set of outputs if a priority queue is presented with permutations from a pattern class $\C$ as its set of inputs.  Usually $\C$ will be defined in terms of its basis and we shall want to find the basis of $\C\A$.
%
%
%

As we shall see $\C\A$ is usually much larger than $\C$.  In the next result we give the precise condition that determines whether $\C=\C\A$.  To state it we recall the weak order on permutations of length $n$ as the transitive closure of permutation pairs $(\alpha, \beta)$ where $\beta$ is obtained from $\alpha$ by interchanging two consecutive values of $\alpha$ in such a way that a new inversion is created.

\begin{theorem}\label{invariant}
Let $\C=\Av(B)$.  Then $\C=\C\A$ precisely when every permutation in the upward weak closure of $B$ contains a permutation
of $B$ as a pattern.
\end{theorem}
\begin{proof}
In the notation of relational composition $\C=\C\A$ if and only if $\C=\C\A^*$ where $\A^*$ denotes the transitive closure of the relation $\A$.  However, by a result of \cite{atkinson:priority-queues-and:}, $\A^*$ is the weak order and so $\C=\C\A$ if and only if $\C$ is closed downwards in the weak order.  But, as shown in \cite{albert:sorting-classes:}, this is equivalent to the condition in the statement of the proposition.
\end{proof}

Put another way, $\C = \C\A$ precisely when $\A = \Av(X)$ for some set $X$ which is upward closed in the weak order.

\begin{corollary}
\label{descending-A}
If $\C=\Av(t\ t-1\ \cdots\ 1)$ then $\C\A=\C$.
\end{corollary}

Sadly, this result is the only case in which we have managed to determine $\Av(\alpha)\A$ with $|\alpha|>3$ (though we have conjectural descriptions of the bases of these classes for all permutations $\alpha$ of length four).  For $\alpha$ of length 2 we already know from the Corollary that $\Av(21)\A=\Av(21)$ while the other length 2 case is useful enough to be recorded explicitly.

\begin{proposition}\label{decreasing-image}
$\Av(12)\A=\Av(132)$.
\end{proposition}
\begin{proof}
When a decreasing sequence is processed by a priority queue the ``remove smallest'' operation becomes ``remove most recently inserted'' so the behavior is just like being processed by a stack.  The result now follows from the theory of stack permutations.
\end{proof}

In Section \ref{one-length-3} we consider the pattern classes $\Av(\alpha)\A$ for each of the 6 permutations of length 3, show that all of them are finitely based and find their bases. In Section \ref{two-length-3}, we give similar results for $\Av(\alpha,\beta)\A$, where $|\alpha|=|\beta|=3$ and briefly comment on some related results.  Section \ref{infinite-section} contains an example to show that, in general, $\Av(\alpha)\A$ is not necessarily finitely based.  The final section considers briefly the case $\A\C$ and discusses a number of open problems.

Our principal tool is a result from \cite{atkinson:the-permutation:}: Proposition \ref{poset-condition} below.  Before stating this result we need to define, for every sequence of distinct integers $\tau$, a poset $P(\tau)$.  The elements of $P(\tau)$ are precisely the elements of the sequence $\tau$ and the order relation $\prec $ is defined by $x\prec y$ if $x$ precedes $y$ in $\tau$ and either
\begin{itemize}
\item $xy\sim 21$, or
\item $xzy\sim 132$ for some element $z$ lying between $x$ and $y$ in $\tau$.
\end{itemize} 

Another way to think of $P(\tau)$ is to write $\tau=\alpha n\beta$, where $n$ is the largest element of $\tau$,  and then the constraints of $P(\tau)$ are
\begin{itemize}
\item $n\prec b$ for all $b\in\beta$,
\item $a\prec b$ for all $a\in\alpha, b\in\beta$, and
\item constraints of $P(\alpha)$ and constraints of $P(\beta)$.
\end{itemize}

{\bf Example} If $\tau=31524$ then $P(\tau)$ has constraints $5\prec 2$, $5\prec 4$, $\{3,1\}\prec\{2,4\}$, and $3\prec 1$.

If $\tau$ is a permutation then a linear extension of $P(\tau)$ can be considered as a sequence in its own right simply by listing its elements from least to greatest with respect to the linear order extending $\prec$. This turns out to be intimately connected with the concept of allowable pair as shown by the following proposition.

\begin{proposition}\label{poset-condition}
$(\sigma,\tau)$ is an allowable pair if and only if $\sigma$ is a linear extension of $P(\tau)$.
\end{proposition}

%

We briefly discuss this result in the context of pattern classes $\C=\Av(\alpha)$.  Suppose $\tau$ is a permutation and that $P(\tau)$ has a chain of the form $a_1\prec  a_2\prec \cdots \prec  a_r$ where $a_1a_2\ldots a_r$ is isomorphic to $\alpha$; for brevity we call this an $\alpha$-chain.  In such a case every linear extension of $P(\tau)$ contains the subsequence $a_1 a_2 \cdots  a_r$ and so none of them are in $\C$.  Therefore $\tau\not\in \C\A$.  So, a necessary condition for $\tau \in \C\A$ is that $P(\tau)$ should contain no $\alpha$-chains. As we shall see in Section \ref{infinite-section} the optimistic hope that this condition is also sufficient often fails.  Nevertheless it does hold (see Section \ref{one-length-3}) when $|\alpha|=3$: in other words we shall prove 

\begin{theorem}\label{chain3}
For any permutation $\alpha$ of length 3,  $\tau \in \Av(\alpha)\A$ if and only if $P(\tau)$ contains no $\alpha$-chain.
\end{theorem}

The proof of this theorem is contained in Section \ref{one-length-3} and consists basically of a case by case analysis of the classes $\Av(\alpha)\A$.

Note that the condition that $P(\tau)$ contains no $\alpha$-chain can be captured by a finite  set of avoidance conditions:

\begin{lemma}\label{chain-constraints} $P(\tau)$ has a chain $a_1\prec a_2\prec \cdots a_k$ if and only if $\tau$ has a subsequence $a_1b_1a_2b_2\cdots a_{k-1}b_{k-1}a_k$ such that
\begin{itemize}
\item if $a_i>a_{i+1}$ then $b_i$ is the empty term,
\item if $a_i<a_{i+1}$ then $b_i>a_{i+1}$
\end{itemize}
\end{lemma}
\begin{proof}
Suppose first that the two conditions hold.  The conditions are precisely those that ensure $a_1\prec  a_2\cdots \prec  a_k$ are constraints of $P(\tau)$.

Conversely suppose that $P(\tau)$ has a chain $a_1\prec  a_2\prec \cdots a_k$.  By definition $a_1, a_2,\ldots, a_k$ appear in this order within $\tau$.  A constraint of the form $a_i\prec a_{i+1}$ with $a_i<a_{i+1}$ can arise only because of some term $b_i>a_{i+1}$ appearing between $a_i$ and $a_{i+1}$.
\end{proof}

In particular it  follows that $P(\tau)$ has no $\alpha$-chain if and only if $\tau$ contains no permutation with the two properties specified in this lemma.  The set of such permutations is easy to compute once $\alpha$ is known.  For example, if $\alpha=123$, we seek permutations $axbyc$ with $a<b<c$, $x>b$ and $y>c$; these are 13254, 14253, 15243.

%

\section{$\Av(\alpha)\A$ when $|\alpha|=3$}
\label{one-length-3}

In this section we verify Theorem \ref{chain3} by a case analysis.  As a consequence of Lemma \ref{chain-constraints} and the remarks that followed it we obtain the basis of $\Av(\alpha)\A$ for each of the 6 permutations $\alpha$ of length 3.  The case $\alpha = 321$ is already covered by Corollary \ref{descending-A}. In each remaining case we obtain the basis of the class $\Av(\alpha)\A$ as a corollary to the main result, using the observation at the end of the preceding section.

Note that the symmetries of the pattern containment order cannot be exploited here since they are not symmetries of the relation $\A$. Essentially this is because the operation of a priority queue depends on both the order in which it receives (and outputs) elements and their relative sizes. No non-trivial symmetry of the pattern containment order respects both these relationships. 

\begin{proposition}
If $P(\tau)$ has no 312-chain then it has a 312-avoiding linear extension.
\end{proposition}
\begin{proof}
Put $\tau=m_1\tau_1m_2\ldots m_k\tau_k$ where $m_1,\ldots,m_k$ are the left-to-right maxima of $\tau$.  Let $\lambda_i$ be the sequence of values in $\tau_i$ but written in decreasing order and put $\lambda=m_1\lambda_1m_2\lambda_2\ldots m_k\lambda_k$.  Then $\lambda$ is the required 312-avoiding linear extension.  To verify that $\lambda$ is a linear extension note that all the constraints of $P(\tau)$ are respected because the only doubt would be over two elements $x,y\in\tau_i$ with $x<y$ and $x\prec y$.  But then $m_i\prec x\prec y$ would be a 312-chain in $P(\tau)$.  Finally $\lambda$ avoids 312 for, if there were a sequence $zxy\sim 312$ in $\lambda$, then $x$ and $y$ would necessarily lie in distinct $\tau_i$ and $\tau_j$. Then, because of the sequence $xm_jy$ in $\tau$ we would have $x\prec y$ and hence $m_i\prec x\prec y$ would be a 312-chain in $P(\tau)$.
%
\end{proof}
\begin{corollary}
$\Av(312)\A=\Av(3142, 4132)$.
\end{corollary}

\begin{proposition}\label{prop132}
If $P(\tau)$ has no 132-chain then it has a 132-avoiding linear extension.
\end{proposition}
\begin{proof}
Suppose that $\tau=\alpha n\beta$ where $n$ is the maximum value occurring in $\tau$. We suppose inductively that the result is true for all permutations of length less than $n$ (observing that the base case of $n = 1$ is trivial).

If $\alpha$ is empty, then we can simply (by induction) take $n$ followed by a 132-avoiding linear extension of $P(\beta)$ to obtain a 132-avoiding linear extension of $\tau$. So, henceforth assume that $\alpha$ is non-empty. Partition the values occurring in $\alpha$ into an increasing sequence of non-empty intervals $X_1$, $X_2$, \ldots, $X_k$, such that if $x_i \in X_i$ and $x_{i+1} \in X_{i+1}$, then some element of $\beta$ lies between them in value. Similarly partition the elements of $\beta$ into intervals $Y_i$ for $0 \leq i \leq k+1$ such that the elements of $Y_i$ lie above $X_i$ and below $X_{i+1}$ (with the obvious modifications for $Y_0$ and $Y_{k+1}$). Note that all the $Y_i$ except possibly $Y_0$ and $Y_{k+1}$ are non-empty.

Consider the constraints of $P(\tau)$ between elements of $\alpha$.  Such constraints  arise precisely from the constraints of $P(\alpha)$.  By induction we may find a 132-avoiding linear extension of $P(\alpha)$ and it will not violate any constraints within the super-poset $P(\tau)$.  In this linear extension let $\lambda_i$ be the subsequence whose values come from $X_i$.

There are no poset constraints of the form $\ell_i\prec \ell_j$ with $i<j$, $\ell_i\in\lambda_i$, $\ell_j\in\lambda_j$ because, with $y \in Y_i$, $\ell_1 \prec \ell_2 \prec y$ would be a 132-chain of $P(\tau)$.  So, in fact, we may take the linear extension of $P(\alpha)$ to have the form $\lambda=\cdots\lambda_3 \lambda_2 \lambda_1$ (in rearranging the $\lambda_i$ in this way we cannot introduce a 132-subsequence since they now form a descending sequence of intervals of values, so the only possible 132 occurrences would be within a single $\lambda_i$ and we know already that this does not take place).


The values of $Y_1\cup Y_2\cup\ldots$ occur in increasing order in $\beta$ (because if two of them $z,y$ say occur in decreasing order then, with $x \in X_1$, we would have a 132-chain $x \prec z \prec y$).  Hence these values also form an antichain of values in $P(\tau)$ because there are no intervening larger values.

Furthermore there is no constraint $y_0 \prec y$ in $P(\tau)$ with $y_0 \in Y_0$ and $y \in Y_i$, $i>0$ as such a constraint could arise only from some intervening term of $Y_1 \cup Y_2 \cup\ldots$ larger than $y$ contradicting the previous observation that the elements of $Y_1 \cup Y_2 \cup \ldots$ occur in increasing order.  Hence, (using induction again) taking $\mu_0$ to be a 132-avoiding linear extension of $P(Y_0)$, and $\mu'$ to be the terms of $Y_1 \cup Y_2 \cup \ldots$ in increasing order, $\mu = \mu' \mu_0$ is a  linear extension of $P(\beta)$ that obviously avoids 132.

Now it follows that $n \lambda \mu$ is a 132-avoiding linear extension of $P(\tau)$.  The check is routine.  Clearly any 132-patterns must lie across $\lambda$ and $\mu$.  But one element in $\lambda$ and two elements in $\mu$ is impossible because the elements on $\mu$ would have to occur in $\mu'$ (being larger than some element of $\alpha$) and $\mu'$ is increasing. Likewise, two elements in $\lambda$ and one in $\mu$ is impossible because the first two elements would have to lie in some common $\lambda_i$ and the third could not separate them by value.
\end{proof}

\begin{corollary}
$\Av(132)\A=\Av(1432)$.
\end{corollary}

\begin{proposition}
If $P(\tau)$ has no 231-chain then it has a 231-avoiding linear extension.
\end{proposition}
\begin{proof}
Again we will proceed inductively.
Let $\tau$ avoid 2431 (the condition that its poset has no 231-chain) and consider its left-to-right maxima. If $\tau$ has only one left-to-right maximum, i.e.~$\tau = n \tau'$ then with $\lambda'$ a 231-avoiding linear extension of $P(\tau')$, we have $n \lambda'$ as a 231-avoiding linear extension of $P(\tau)$.  So, assume henceforth that $\tau$ has at least two left-to-right maxima, and let the values of the left-to-right maxima be denoted $m_1, m_2, \dots, m_k$.

The non-left-to-right maxima fall into layers of values between successive maxima and these sets of values occur left to right in $\tau$ (from the 2431 condition).  The situation is illustrated in Figure \ref{fig-2431} where the grey boxes represent subsequences $\theta_1,\theta_2,\ldots\theta_k$ of $\tau$.  While the values in each $\theta_i$ are, by definition, contiguous they are not necessarily contiguous by position since they may be punctuated by left-to-right maxima lying above them.

\begin{figure}[h]
\begin{center}
\psset{xunit=0.05cm, yunit=0.05cm}
\psset{linewidth=0.002cm}
\newgray{gr}{0.9}
\begin{pspicture}(0,0)(100,65)
\psline[linecolor=black,linestyle=solid](10,10)(100,10)
\psline[linecolor=black,linestyle=solid](20,20)(100,20)
\psline[linecolor=black,linestyle=solid](30,30)(100,30)
\psline[linecolor=black,linestyle=solid](40,40)(100,40)
\psline[linecolor=black,linestyle=solid](50,50)(100,50)
\psline[linecolor=black,linestyle=solid](60,60)(100,60)
\pscircle*(10,10){0.05cm}
\pscircle*(20,20){0.05cm}
\pscircle*(30,30){0.05cm}
\pscircle*(40,40){0.05cm}
\pscircle*(50,50){0.05cm}
\pscircle*(60,60){0.05cm}
\pspolygon*[linearc=0.1,linecolor=gr](11,1)(24,1)(24,9)(11,9)
\pspolygon[linearc=0.1](11,1)(24,1)(24,9)(11,9)
\pspolygon*[linearc=0.1,linecolor=gr](25,11)(42,11)(42,19)(25,19)
\pspolygon[linearc=0.1](25,11)(42,11)(42,19)(25,19)
\pspolygon*[linearc=0.1,linecolor=gr](44,21)(54,21)(54,29)(44,29)
\pspolygon[linearc=0.1](44,21)(54,21)(54,29)(44,29)
\pspolygon*[linearc=0.1,linecolor=gr](55,31)(75,31)(75,39)(55,39)
\pspolygon[linearc=0.1](55,31)(75,31)(75,39)(55,39)
\pspolygon*[linearc=0.1,linecolor=gr](76,41)(84,41)(84,49)(76,49)
\pspolygon[linearc=0.1](76,41)(84,41)(84,49)(76,49)
\pspolygon*[linearc=0.1,linecolor=gr](85,51)(99,51)(99,59)(85,59)
\pspolygon[linearc=0.1](85,51)(99,51)(99,59)(85,59)
\end{pspicture}
\end{center}
\caption{The generic structure of a 2431-avoiding permutation}
\label{fig-2431}
\end{figure}
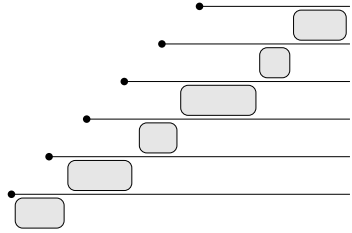

Put $\tau=\tau_1 \tau_2 \ldots \tau_k$ where each $\tau_i$ consists of $\theta_i$ together with its punctuating left-to-right maxima (if any); include in $\tau_i$ any left-to-right maxima that immediately precede $\theta_i$.  The constraints of $P(\tau)$ have two possible forms: they are constraints of $P(\tau_i)$ or of the type $t_i\prec t_j$ with  $t_i\in\tau_i$, $t_j\in\tau_j$ and $i<j$.  Notice that none of these constraints have the form $t\prec m_i$.

By induction we may find a 231-avoiding linear extension of each $P(\tau_i)$ (unless $\tau_1 = \tau$:-  in this case take a 231-avoiding linear extension of $\tau$ with $m_1$ deleted).  In each such linear extension let $\lambda_i$ denote the sequence of values belonging to $\theta_i$; this too avoids 231.

Consider $m_km_{k-1}\cdots m_1\lambda_1\lambda_2\cdots$.  This  is a linear extension of $P(\tau)$ that also avoids 231.
%
%
%
%
\end{proof}

\begin{corollary}
$\Av(231)\A=\Av(2431)$.
\end{corollary}

\begin{proposition}
If $P(\tau)$ has no 213-chain then it has a 213-avoiding linear extension.
\end{proposition}
\begin{proof}

Put $\tau=\alpha n\beta$ where $n$ is the maximum value that occurs in $\tau$. As always we proceed by induction, and as always, the case where $\alpha$ is empty is trivial.

Similarly to the notation of Proposition \ref{prop132} we partition the set of values  occurring in $\alpha$ and $\beta$ into intervals
$X_0, Y_1,X_1,Y_1,\ldots,X_k,Y_k,X_{k+1}$ where each member of this list is a set of values smaller than its successor in the list and where the elements of $X_i$ are terms of $\alpha$ and the elements of $Y_i$ are terms of $\beta$.  These sets are non-empty with the possible exception of $X_0$ and $X_{k+1}$.

The set of values from $X_0 \cup X_1 \cup \cdots \cup X_{k}$ occur in increasing order in $\tau$ because if there were a decreasing pair $vu$ then, with $w \in Y_k$, $v \prec u \prec w$ would be a 213-chain in $P(\tau)$.  We write $\alpha=\phi_0 \theta_1 \phi_1 \theta_2 \cdots \theta_r \phi_r$ where $\theta_1\theta_2\cdots\theta_r$ is the increasing set of values from $X_0 \cup \cdots \cup X_{k}$ and $\phi_0 \phi_1\cdots \phi_r$ are separating sequences of values comprising the set $X_{k+1}$ (all non-empty except possibly for $\phi_0$ and $\phi_r$).
The constraints of $P(\tau)$ between the elements of $X_0 \cup X_1 \cup \cdots \cup X_{k}$ are easily seen to be only of the form $a\prec b$ where $a \in \theta_i$, $b \in \theta_j$ and $i<j$. There are no constraints between elements of any $\theta_i$.  Moreover, the set $T_i$ of elements of $\theta_i$ forms a poset interval in $P(\alpha)$.

Now, by induction, we can find some 213-avoiding linear extension of $P(\alpha)$.  Then we can arrange the values of each poset interval $T_i$ in increasing order without introducing a 213-subsequence obtaining another linear extension  $\lambda$.  Note that, in $\lambda$, all elements of $X_0 \cup X_1 \cup \cdots \cup X_{k}$ occur in increasing order.

Next observe that there is no constraint of $P(\tau)$ of the form $u \prec w$ with $u \in Y_i$, $w \in Y_j$ and $i<j$; for then, with $v \in X_i$, $v \prec u \prec w$ would be a 213-chain of $P(\tau)$.  Let $\beta_i$ be the sequence of values of $Y_i$ as they occur in $\beta$.  By induction we can find 213-avoiding linear extensions $\mu_i$ of each $P(\beta_i)$ and then $\mu_k \cdots \mu_2 \mu_1$ will be a 213-avoiding linear extension of $P(\beta)$.

Finally, consider the permutation $n \lambda \mu$.  This is certainly a linear extension of $P(\tau)$.  Furthermore it has no 213-subsequence.  Clearly there can be no 213-subsequence containing $n$, nor contained entirely within $\lambda$ or within $\mu$.  There cannot be a 213-subsequence $bac$ with both $b, a \in \lambda$ because $X_0 \cup X_1\cup \cdots \cup X_{k}$ occur in increasing order and so one of $b,a$ would have to lie in $X_{k+1}$ (and then there is no larger element in $\beta$).  Nor could we have $b\in \lambda$ and $a,c \in \mu$ because then $a,c$ would need to lie in some common $Y_i$ and $b$ could not separate them by value.
\end{proof}

\begin{corollary}
$\Av(213)\A=\Av(2143)$.
\end{corollary}

\begin{proposition}
If $P(\tau)$ has no 123-chain then it has a 123-avoiding linear extension.
\end{proposition}
\begin{proof}
Again we let $\tau=\alpha n\beta$ where $n$ is the maximum element occuring in $\tau$, and again we may as well assume that $\alpha$ is non-empty.  By induction we can find 123-avoiding linear extensions of $P(\alpha)$ and $P(\beta)$ and so we can find linear extensions of $P(\tau)$ which consist of $n$ followed by 123-avoiding linear extensions of the entries of $\alpha$ and the entries of $\beta$.  We will argue that for any linear extension, $\epsilon = n \lambda \mu$, of this type containing a 123-subsequence there is another linear extension of the same type that has more inversions than $\epsilon$ does. 

Suppose that $\nu$ has a 123 pattern $xyz$. Either $x,y \in \lambda$ and $z \in \mu$ or $x \in \lambda$ and $y,z \in \mu$.  Suppose that the former occurs and, for the given $z$, take $x,y$ as close in position as possible.

We know that $x \not \prec  y$ (for certainly $y \prec  z$ and there is no 123-chain in $P(\tau)$).  If $x$ and $y$ are adjacent in $\lambda$  then because $x \not \prec  y$ we can exchange $x$ and $y$ and get a new linear extension of the same type with one more inversion.

So suppose that there is some intervening element $w$ occurring immediately after $x$ in $\lambda$ where, by choice of $x,y$, we shall have $w>y$.  If $x \not \prec  w$ we may exchange $x$ and $w$ and get a new linear extension of the same type with one more inversion.

On the other hand, it is not possible that $x \prec  w$. If it were then we would have, in $\alpha$, a subsequence $xtw$ for some $t>w$.  Now, certainly $y$ succeeds $x$ in $\alpha$ (or we would have $y \prec  x$ and $y,x$ would occur in this order in $\lambda$) but $y$ cannot occur after $t$ in $\alpha$ (that would mean, because of the sequence $xty$, that $x \prec y$) and it cannot occur between $x$ and $t$ in $\alpha$ for then the sequence $ytw$ in $\alpha$ would mean that $y\prec w$  contradicting that $w$ precedes $y$ in $\lambda$.

Exactly the same argument applies to $xyz$ sequences with $x\in\lambda$ and $y,z\in\mu$. So, if we take $\epsilon = n \lambda \mu$ to be that linear extension of $P(\tau)$ having the maximum possible number of inversions, it must avoid 123.

\end{proof}

\begin{corollary}
$\Av(123)\A=\Av(13254, 14253, 15243)$.
\end{corollary}

%
%
%

\section{$\Av(\alpha,\beta)\A$ when $|\alpha|=|\beta|=3$}
\label{two-length-3}

There are 15 pairs $(\alpha,\beta)$ of permutations of length 3.  The structure of each of the pattern classes they define is well-known but we do not know of any convenient reference to that structure.  So, for completeness, we list in Table \ref{table-two-three} descriptions of these 15 pattern classes $\C$ leaving the elementary justifications to the reader. The ordering of the rows in that table groups similarly structured classes together.  The notation used is fairly standard.  $\I$ and $\D$ denote the classes of increasing and decreasing permutations respectively.  The sum of two permutations $\sigma\oplus\tau$ is that permutation $s_1\ldots s_pt_1\ldots t_q$ where $s_1\ldots s_p\sim \sigma$, $t_1\ldots t_q\sim\tau$ and $s_i<t_j$ for all $i,j$.  Skew sums, $\sigma \ominus \tau$, are defined similarly except the final property is $s_i>t_j$ for all $i,j$.  Notation such as $213[\I,\I,\I]$ denotes the class of all inflations of 213 in which each term is replaced by an increasing sequence of consecutive values.

\newcommand{\pb}[1]{\parbox[t]{4cm}{\raggedright #1 \strut}}

\begin{table}
\begin{tabular}{ccll}
\hline
Index & Basis of $\C$ & Structure of $\C$ & Basis of $\C\A$\\
\hline
1&132, 321&$213[\I,\I,\I]$&321, 2143, 2413\\
2&213, 321&$132[\I,\I,\I]$&321, 2143, 2413\\
3&231, 312&\pb{Sums of decreasing permutations}&2413, 2431, 3142, 4132\\
4&231, 321&\pb{Sums of $t\ 1\ 2\ \cdots t-1$}&231,321\\
5&312, 321&\pb{Sums of $2\ 3\ \cdots t\ 1$}&312, 321\\
6&123, 231&$312[\D,\D,\D]$&\pb{2431, 13254, 13524, 14253, 15243, 31524, 461325}\\
7&123, 312&$231[\D,\D,\D]$&\pb{3142, 4132, 13254,13524,13542}\\
8&132, 213&\pb{Skew sums of increasing permutations}&\pb{1432, 2143,13524}\\
9&123, 132&\pb{Skew sums of $t-1\ t-2\ \cdots1\ t$}&\pb{1423, 1432, 13254}\\
10&123, 213&\pb{Skew sums of $1\ t\ t-1\ \cdots 2$}&\pb{1243, 2143}\\
11&132, 231&\pb{Permutations whose diagram is shaped like $\vee$}&\pb{1432, 2431}\\
12&132, 312&\pb{Permutations whose diagram is shaped like $<$}&\pb{132}\\
13&213, 231&\pb{Permutations whose diagram is shaped like $>$}&\pb{2143, 2413, 2431}\\
14&213, 312&\pb{Permutations whose diagram is shaped like $\wedge$}&\pb{2143, 3142, 4132}\\
15&123, 321&Finite&\pb{321,1423,
    2314,
    2341,
    4123,
    12345,
    12354,
    12435,
    12453,13245,
    13254,
    21345,
    21354,
    21435,
    21453,
    31245,
    3125}\\
\hline
\end{tabular}
\caption{The structure of classes $\C$ having two basis elements of length 3, and the corresponding bases of the classes $\C\A$.}
\label{table-two-three}
\end{table}

We now justify the final column of Table \ref{table-two-three} treating each case in turn.  In all cases it is routine to verify that none of the claimed basis permutations $\beta$ lie in $\C\A$. We simply have to verify that,  if $\beta$ is a proposed basis element, then all linear extensions of $P(\beta)$ involve one of the basis permutations of $\C$. The proposed bases were in fact produced by computer search precisely with respect to that property.  Thus we shall only be concerned with proving that a permutation $\tau$ that avoids these basis permutations is an image of some permutation in $\C$.

\case{1}{$\Av(132, 321)\A$}

Permutations $\tau\in\Av(321, 2143, 2413)$ are merges of two increasing sequences $L$ and $U$ (with $L<U$) (this is a symmetry of a result in \cite{atkinson:restricted-perm:}).  It is readily checked that $P(\tau)$ has only constraints $u\prec\ell$ with $u\in U$, $\ell\in L$ and constraints $\ell_1\prec\ell_2$ with $\ell_1,\ell_2\in L$ and $\ell_1<\ell_2$.  Therefore $P(\tau)$ has a linear extension in which the terms of $L$ and the terms of $U$ occur in the same order as they do in $\tau$, and all the terms of $L$ precede all the terms of $U$;  such permutations lie in  $\Av(132, 321)$ (indeed they lie in $\Av(132,213,321)$).

\case{2}{$\Av(213, 321)\A$}

The proof is exactly like the previous one.

\case{3}{$\Av(231, 312)\A$}

When a permutation that is a sum of decreasing permutations is processed by a priority queue the possible outputs are precisely sums of permutations in $\Av(132)$ (see Proposition \ref{decreasing-image}).  But, from a (symmetry of a) result of \cite{atkinson:restricted-perm:wreath}, the class of all such permutations is $\Av(2413,2431,3142,4132)$.

\case{4 and case 5}{$\Av(231, 321)\A$ and $\Av(312,321)\A$}

These follow from Theorem \ref{invariant}.

%

\case{6}{$\Av(123,231)\A$}

Let $\tau\in\Av(2431, 13254, 13524, 14253, 15243, 31524, 461325)$.  We divide the elements of $\tau$ into three sets:

\begin{itemize}
\item $A$ is the set of  elements that play the role of `1' in a 132-subsequence of $\tau$ together with all  elements less than any of these,
\item $B$ is the set of  elements that play the role of `2' in a 132-subsequence of $\tau$,
\item $C$ is the set of any remaining elements.
\end{itemize}

The goal is to show that $\gamma\alpha\beta$ is a linear extension of $P(\tau)$ where $\alpha$, $\beta$, and $\gamma$ are decreasing permutations whose elements are equal to the elements in $A$, $B$ and $C$ respectively, and that $\gamma \alpha \beta \in 312[\D, \D, \D]$.

First we will verify that $A\cup B\cup C$ is a partition of the terms of $\tau$.  It suffices to prove that $A\cap B$ is empty. 
It is easy to check that no element can both be a `1' in a 132 and a `2' in a 132 (take the copy in which it is a `2', append the 32 of the copy in which it is a `1', and one of the  basis elements 13254, 14253, or 15243 results). The remaining case is that of an element $a$, which is a `2' in a 132 that lies below some `1' of a 132. Call that 132 $bdc$. Then, since $a$ is not a `1', and 2431 is a basis element, we must have $b,d,a,c$ occurring in this order. Now let $xya$ be a 132. In all cases the set of elements $\{b,d,c,x,y,a\}$ contains a subsequence matching one of the basis elements.

Next we prove that, for all $a\in A, b\in B, c\in C$, we have $a<b<c$.  To do this we merely have to prove that if $b \in B$, $b' < b$, $b' \not\in A$, then $b' \in B$.
So let $xyb \sim 132$. Since $b' \not \in A$, $b' > x$ and $b'$ follows $y$ in $\tau$ (else $b'yb \sim 132$). But then $xyb' \sim 132$, so $b' \in B$.

Now define $\lambda=\gamma\alpha\beta$ where $\alpha,\beta, \gamma$ are the terms of $A,B,C$ respectively each written in decreasing order.  So $\lambda\in\Av(123,231)$.  We shall verify that $\lambda$ is a linear extension of $P(\tau)$.  

To begin with notice that, in $\tau$, all elements of $A$ precede all elements of $B$.  For suppose $b\in B$ because of some subsequence $xyb\sim 132$ of $\tau$.  A following element of $A$ could not be smaller than $x$ (else we obtain a subsequence $2431$) nor larger than $b$ (else $b \in A \cap B$), nor  between $x$ and $b$ (else it belongs to $B$). So, no such element can exist.  In particular there are  no constraints $a\prec b$ with $a\in A, b\in B$ and so having $\beta$ following $\alpha$ violates no constraints.

Rather more easily there are no constraints $a\prec c$ or $b\prec c$ for any $a\in A, b\in B, c\in C$ for in both cases this could only happen if $c$ was a `2' in a 132-subsequence.

Finally no constraints between elements of $A$, or between elements of $B$, or between elements of $C$ are violated.  Again for $A$ and $C$ this is clear: if $a_1\prec a_2$ with $a_1<a_2$ then $a_2$ would be a `2' in some 132-subsequence $a_1xa_2$.  But also $B$ has no $b_1\prec b_2$ constraint with $b_1<b_2$.  For if there was a subsequence $b_1xb_2\sim 132$ then $b_1\in A$ which is a contradiction.

\case{7}{$\Av(123,312)\A$}


Let $\tau \in \Av(3142, 4132, 13254, 13524, 13542)$ be given. Let $T$ be the set of elements that play the role of a `1' in a 132-subsequence of $\tau$ and let $S$ be the set of remaining elements of $\tau$.  Clearly $\tau |_S$ has no 132-subsequence.  But $\tau |_T$ has no 132-subsequence either for if $t_1t_3t_2$ was such a sequence we can find $v,u$ with $t_3vu\sim 132$.  Then we have one of $t_1t_3vut_2$, $t_1t_3vt_2u$, $t_1t_3t_2vu$ which have the patterns 13542, 13542, 13254 all of which are forbidden.

Furthermore the set of values of $T$ is a consecutive set.  To see this let $t,t'\in T$ with $t<t'$.  Consider some $x$ with $t<x<t'$.  Suppose first that $t$ precedes $t'$ in $\tau$ and let $t'vu\sim 132$.  If $x$ precedes $v$ then $xvu\sim 132$ and so $x\in T$.  Otherwise $\tau$ has a subsequence $tt'vxu\sim 13524$ or $tt'vux\sim 13542$ both of which are forbidden.  Now suppose $t'$ precedes $t$ in $\tau$.  There is some $v,u$ with $tvu\sim 132$.  Then $t'>u$ (else $tt'vu\sim 3142$ or $tt'vu\sim 4132$).  If $x$ precedes $v$ then, because of $xvu$, $x\in T$.  But otherwise the sequence $t'tvx\sim 3142$ which is forbidden.

The poset $P(\tau)$ has no constraints $s\prec t$ with $s\in S$ and $t\in T$.  Such constraints could only arise if either
\begin{itemize}
\item $s>t$ and $s$ precedes $t$ in $\tau$.  But then there is some $tvu\sim 132$ and then either $stvu\sim 4132$ or $stvu\sim 3142$ or $s$ is smaller than both $v,u$ and hence $s\not\in S$.
\item $s<t$ and there is some $swt$ with $w>t$.  But then, again, $s\not\in S$.
\end{itemize}

Now we construct a linear extension of $P(\tau)$ in $\Av(123, 312)$.  We take the set of values of $S$ and $T$ in decreasing order and place the values of $T$ before the values of $S$.  No poset constraints are violated because there can be none among the elements of $S$ of the form $s_1\prec s_2$ with $s_1<s_2$ for they could only arise from some $y$ with $s_1ys_2\sim 132$ which would mean $s_1\not\in S$, or among the elements of $T$ of the form $t_1\prec t_2$ with $t_1<t_2$ since we know $T$ has no 132-sequence, or across $S$ and $T$ of the form $s\prec t$ as proved above.

\case{8}{$\Av(132,213)\A$}

It is clear that every priority queue computation on a permutation which is a skew sum of $k$ increasing subsequences must map it to a permutation where the increasing subsequences are merged together so that we have $k$ `bands' (sets of consecutive values) of increasing elements. Moreover, no elements taken from 3 distinct bands can form a 132 pattern.  Conversely, if $\tau$ is a permutation of the latter type, it is easy to see that $P(\tau)$ can have no constraints $x\prec y$ where $x<y$ and $x$ and $y$ are in different bands. For then there would be some subsequence $xzy\sim 132$;  but $z$ would lie in a higher band than $y$ and so we would have a 132 pattern between three elements of distinct bands.  Thus $\Av(132,213)\A$ is exactly the class of permutations which can be divided into increasing bands with no 132 pattern between three distinct bands.  By inspection, none of 1432,2143,13524 have this property.  On the other hand let $\pi\in\Av(1432,2143, 13524)$ and consider a partition of the terms of $\pi$ into the smallest possible number of increasing bands.  Suppose there is a subsequence $xzy\sim 132$ with $x,y,z$ in distinct bands.  Then, because of the restrictions imposed by the basis elements, it is easy to check that the set of elements in all the bands between (and including) those containing $x$ and $y$ form an increasing set contradicting the minimality assumption.  Therefore $\Av(132,213)\A=\Av(1432,2143,13524)$.

\case{9}{$\Av(123,132)\A$}

Let $\tau \in \Av(1423, 1432, 13254)$ be given. 
Let $R$ be the set of elements that play the role of a `2' in a 132-subsequence of $\tau$, and let $S$ be the set of remaining elements of $\tau$. Clearly $\tau |_S$ has no 132-subsequence, but also $\tau |_R$ avoids 132.
For suppose that $\tau |_R$ contained a subsequence $ftm\sim132$.  Since $t \in R$  there is a subsequence $axt \sim 132$. Then $x$ must precede $f$ else $fxtm\sim 1432$.  Also, since $f\in R$, there is a subsequence  $bpf\sim 132$. Then $p < t$ else $bftm\sim 1432$, and $p > m$ else $bpftm\sim13254$. But, if $p$ precedes $x$ then $bpxf\sim 1432$, and if $p$ follows $x$ then $bxpt\sim 1423$. So, we have a contradiction in any case, and thus $\tau |_R$ must avoid 132.

For each $r\in R$ let $A_r = \{ a \mid  ayr \sim 132 \mbox{ for some  subsequence }ayr\mbox{ of }\tau \}$.
Then we have

\begin{itemize}
\item
$A_r \subseteq S$, and 
\item for all $b$ if $a < b < r$ and $a \in A_r$ then $b \in A_r$.
\end{itemize}

For the first part, suppose that $xza, ayr \sim 132$. If $z < r$ then $xzays \sim 13254$, if $r < z < y$ then $xzas \sim 1423$ and if $y < z$ then $xzyr \sim 1432$ yielding a contradiction in each case.

For the second part let $ayr \sim 132$, and $a < b < r$. If $b$ follows $r$ then $ayrb \sim 1432$, and if $b $ follows $y$ but precedes $r$ then $aybr \sim 1423$. So, $b$ precedes $y$, hence $byr \sim 132$ and $b \in A_r$.

It follows that if $a \in A_r$, and $a \in A_u$ then $r = u$ (otherwise we would have $r \in A_u$ or vice versa from the second part, contradicting the first part).

Construct $\gamma \in \Av(123,132)$ as follows: begin with the elements of $S$ in descending order and, for each $r \in R$ place it immediately following the smallest element of $A_r$.

To complete the proof we must show that $\gamma$ is a linear extension of $P(\tau)$. In other words we must show that for, every $x\prec y$ in $P(\tau)$, $x$ precedes $y$ in $\gamma$. Now $x\prec y$ arises either because $xy$ is an inversion of $\tau$ or because $xzy\sim 132$ is a subsequence of $\tau$.  In the latter case we have $x\in A_y$ and, as $y$ follows the elements of $A_y$ in $\gamma$, $y$ follows $x$ as required. The former case likewise holds easily if $x$ and $y$ both belong to $R$ or both belong to $S$. Suppose that $y \in R$ and $x \in S$. Let $a$ be the least element of $A_y$. Then $a < y < x$, so $x$ precedes $a$ and hence $y$ in $\gamma$. Finally, suppose that $y \in S$ and $x \in R$. Since $x > y$, and $x$ precedes $y$ in $\tau$, $y \not \in A_x$, and in fact $y < a$ for all $a \in A_x$. By construction, $y$ follows all such $a$ in $\gamma$, and hence also follows $x$.


\case{10}{$\Av(123,213)\A$}

Let $\tau \in \Av(1243,2143)$ be given. 
Let $A$ be the set of elements that play the role of a `1' in a 132-subsequence of $\tau$, and let $B$ be the set of remaining elements of $\tau$. Clearly $\tau |_B$ has no 132-subsequence, and by 1243-avoidance, $\tau |_A$ is decreasing.

For each $a\in A$ let $B_a = \{ b \mid ayb \sim 132\mbox { for some }y \}$.  Observe that

\begin{itemize}
\item
$B_a \subseteq B$ since $\tau |_A$ is decreasing, and
\item For all $b'$ if $a < b' < b$ and $b \in B_a$ then $b' \in B_a$. To prove this choose $y$ so that $ayb \sim 132$. We see that $b'$ cannot precede $a$ (or $b' a y b \sim 2143$), nor is positioned between $a$ and $y$ (or $a b' y b \sim 1243$), so $b'$ follows $y$ and thus $b' \in B_a$.
\end{itemize}

It follows that if $b \in B_a$, and $b \in B_{a'}$ then $a = a'$ (otherwise, from the second part, we would have one of $a' \in B_a$ or $a\in B_{a'}$ which would contradict the first part).

Construct $\gamma \in \Av(123,213)$ as follows: begin with the elements of $B$ in descending order and, for each $a \in A$, place $a$ immediately preceding the largest element of $B_a$.

To complete the proof we must show that $\gamma$ is a linear extension of $P(\tau)$. As in the previous case we must show that, for every $x\prec y$ in $P(\tau)$, $x$ precedes $y$ in $\gamma$.  Again $x\prec y$ arises either because  $xy$ is an inversion in $\tau$ or because $xzy\sim 132$ is a subsequence of $\tau$.  In the latter case $y\in B_x$ and, by construction, $x$ precedes all elements of $B_x$ in $\gamma$. The former likewise holds easily if $x$ and $y$ both belong to $A$ or both belong to $B$. Suppose that $x\in A$ and $y \in B$. Since $x>y$, $x$ precedes $y$ in $\gamma$ as required.  On the other hand suppose that $y \in A$ and $x \in B$. Then, since $x>y$ but $x$ precedes $y$ in $\tau$, $x \not \in B_y$. Therefore $x$ is greater than the largest element of $B_y$ and precedes $y$ in $\gamma$. 


\case{11}{$\Av(132,231)\A$}

Let $\tau\in\Av(1432,2431)$ be given.  Let $T$ be the set of elements that play the role of `2' in a 132-subsequence of $\tau$
and let $S$ be the remaining elements of $\tau$.  Then observe
\begin{itemize}
\item
There are no constraints $t\prec s$ with $t\in T$ and $s\in S$.  Such a constraint can arise either because $t>s$ and $t$ precedes $s$ in $\tau$ (and then, as $t\in T$, we can find a subsequence $xyt\sim 132$ in which case $xyts\sim 1432$ or $xyts\sim 2431$ which are forbidden) or because there is a subsequence $tvs\sim 132$ of $\tau$ contradicting that $s\not\in T$.
\item
There are no constraints $s_1\prec s_2$ with $s_1<s_2$ between elements of $S$.  This could only occur if there was a subsequence $s_1vs_2\sim 132$ of $\tau$ contradicting that $s_2\not\in T$.
\item
There are no constraints $t_1\prec t_2$ with $t_1>t_2$ between elements of $T$.  If so, $t_1\in T$ implies that there is a subsequence $abt_1\sim 132$ and then the subsequence $abt_1t_2$ is either isomorphic to 1432 or 2431.
\end{itemize}

Now define $\sigma$ to consist of the elements of $S$ in decreasing order followed by the elements of $T$ in increasing order.  Since no constraints are violated this permutation is a linear extension of $P(\tau)$ and, by construction, it lies in $\Av(132,231)$.

\case{12}{$\Av(132,312)\A$}

Any permutation in $\Av(132)$ is, by  Proposition \ref{decreasing-image}, the image of some decreasing permutation, and decreasing permutations are members of $\Av(132, 312)$.

\case{13}{$\Av(213,231)\A$}

Permutations of $\Av(2143, 2413, 2431)$ are easily seen to be  a merge of two sequences $\lambda$ and $\mu$ where $\lambda$ avoids 132, $\mu$ is increasing, and all terms of $\mu$ are less than all terms of $\lambda$.  A pre-image in $\Av(213,231)$ is found as follows.  Consider a priority queue algorithm that generates $\lambda$ from a decreasing sequence: a series of input and delete operations.  We insert into this operation sequence insert-delete pairs that generate the interpolated elements $\mu$ as and when they are needed.  Because $\mu<\lambda$ a newly inserted $m\in\mu$ will be the minimal entry of the priority queue and will then be immediately deleted. 

%
%


\case{14}{$\Av(213, 312)\A$}

Permutations of $\Av(2143, 3142, 4132)$ are easily seen to have the form $\lambda\mu$ where $\lambda$ is increasing and $\mu$ avoids 132.  Hence we can find a pre-image of the form $\lambda\mu'$ where $\mu'$ is a decreasing pre-image of $\mu$.

\case{15}{$\Av(123,321)\A$}

This is a routine check of finitely many permutations.

For three or more basis elements we merely state the results for precisely three basis elements and only for the infinite  classes $\Av(\alpha,\beta,\gamma)$.  The proofs are routine.

\centerline{
\begin{tabular}{lll}
\hline
\multicolumn{1}{c}{Basis of $\C$}&\rule{10pt}{0pt}&\multicolumn{1}{c}{Basis of $\C\A$}\\
\hline
123, 132, 213&&1243, 1423, 1432, 2143\\
123, 132, 231&&1423, 1432, 2431, 13254, 461325\\
123, 132, 312&&132\\
123, 213, 231&&1243, 2143, 2413, 2431\\
123, 213, 312&&1243, 2143, 3142, 4132\\
123, 231, 312&&2413, 2431, 3142, 4132, 13254\\
132, 213, 231&&1432, 2143, 2413, 2431\\
132, 213, 312&&132\\
132, 213, 321&&321, 2143, 2413\\
132, 231, 312&&132\\
132, 231, 321&&231, 321, 2143\\
132, 312, 321&&132, 312, 321\\
213, 231, 312&&2143, 2413, 2431, 3142, 4132\\
213, 231, 321&&213, 231, 321\\
213, 312, 321&&312, 321, 2143\\
231, 312, 321&&231, 312, 321\\
\end{tabular}
}

\section{An infinitely based $\Av(\alpha)\A$}
\label{infinite-section}

This section is devoted to proving

\begin{theorem}
$\Av(2431)\A$ is not finitely based.
\end{theorem}

We consider the infinite family
%
%
%
\[2\ m+1\ (4\ 1)\ (6\ 3)\ (8\ 5)\ (10\ 7)\ldots (m-4\ m-7)\ m-2\ m\ m-5\ m-1\ m-3\]
for even $m\geq 4$.  The parentheses simply indicate how members of this family are formed.  A typical member of the family is
\[2\ 13\ 4\ 1\ 6\ 3\ 8\ 5\ 10\ 12\ 7\ 11\ 9\]
We shall prove that 
\begin{itemize}
\item no permutation in this family lies in $\Av(2431)\A$, and
\item all proper subpermutations of such permutations lie in $\Av(2431)\A$.
\end{itemize}


To see the first of these consider a typical member $\tau$ of the family.  We have to prove that every linear extension of $P(\tau)$ contains 2431.   For a contradiction suppose that $\lambda$ is a linear extension containing no 2431 pattern.

In $P(\tau)$ we have  $m\prec m-1\prec m-3$ so, in $\lambda$, $m-2$ cannot precede $m$ (or $m-2\ m\ m-1\ m-3$ would be a 2431 pattern).  Therefore (as $m-2\prec m-5$ in $P(\tau)$) $\lambda$ has a subsequence $m\ m-2\ m-5$.  But this means that $m-4$ cannot precede $m$ in $\lambda$ (or $m-4\ m\ m-2\ m-5$ would be a 2431 pattern).  But, again, $m-4\prec m-7$ means that $\lambda$ has a subsequence $m\ m-4\ m-7$ and hence $m-6$ cannot precede $m$ in $\lambda$ because it would give a 2431 pattern $m-6\ m\ m-4\ m-7$.  Continuing to argue in this way we eventually conclude that $m$ must precede 4.  However $2\prec 4\prec 1$ and $2\prec m$ and this implies that $2\ m\ 4\ 1$ is a subsequence of $\lambda$, a contradiction.

The second thing we have to prove is that, if $\tau'$ is the result of removing an arbitrary point of $\tau$ then $P(\tau')$ has a 2431-avoiding linear extension (so that $\tau'\in \Av(2431)\A$).

For illustrative purposes we take
 \[\tau=2\ 13\ 4\ 1\ 6\ 3\ 8\ 5\ 10\ 12\ 7\ 11\ 9\]
 and consider the following linear extensions of $P(\tau)$
 
\[13\ 2\ 4\ 1\ 6\ 3\ 8\ 5\ 10\ 12\ 7\ 11\ 9\]
\[13\ 2\ 4\ 1\ 6\ 3\ 8\ 5\ 12\ 10\ 7\ 11\ 9\]
\[13\ 2\ 4\ 1\ 6\ 3\ 12\ 8\ 5\ 10\ 7\ 11\ 9\]
\[13\ 2\ 4\ 1\ 12\ 6\ 3\ 8\ 5\ 10\ 7\ 11\ 9\]
\[13\ 2\ 12\ 4\ 1\ 6\ 3\ 8\ 5\ 10\ 7\ 11\ 9\]
 which are derived from $\tau$ itself by interchanging the first two terms and by changing the position of 12 so that it is successively before 7, 10, 8, 6, 4.  These linear extensions each have a unique subsequence isomorphic to 2431 (respectively 10 12 11 9, 8 12 10 7, 6 12 8 5, 4 12 6 3 and 2 12 4 1).
 
 In general any constraint $x\prec y$ of $P(\tau-k)$ is also a constraint of $P(\tau)$ and therefore if we remove a point $k$ from a linear extension of $P(\tau)$ we shall obtain a linear extension of $P(\tau-k)$.  For any $k<13$ there is at least one of the linear extensions in the above list which, when $k$ is removed, does not contain 2431 yielding a linear extension of $P(\tau-k)$ that does not contain 2431.  For $k=13$ the linear extension of $12\ 2\ 4\ 1\ 6\ 3\ 8\ 5\ 10\ 7\ 11\ 9$ does not contain 2431.
 
 The general case is completely similar.

\section{Conclusions and open questions}

Most of our paper has been about $\C\A$ where $\C$ is a pattern class with basis elements of lengths 3.  It seems to be very difficult to describe $\C\A$ for an arbitrary pattern class.  However, it is possible that, at least for principal pattern classes $\C=\Av(\alpha)$, it might be possible to solve the finite basis question in general.  We have run some computer experiments in the case $|\alpha|=4$ which seem to indicate that most  of these 24 classes are finitely based.  However we are very far from  finding any general necessary and sufficient conditions for $\Av(\alpha)\A$ to be finitely based.

There is, of course, the dual problem: given a pattern class $\C$ determine the pattern class $\A\C$ which is the set of permutations that a priority queue can transform into a permutation of $\C$.  If $\C$ contains every increasing permutation then $\A\C$ is the set of all permutations so the problem is only interesting if the basis of $\C$ contains $\iota_k=12\cdots k$ for some $k$.  For $\Av(\iota_k)$ itself the class of permutations that can transform to it is exactly $\Av(\iota_k)$.  This is because every permutation of $\Av(\iota_k)$ can transform to itself.  But a permutation not in $\Av(\iota_k)$ contains an increasing subsequence of length $k$ and this subsequence is transformed without change.  These remarks give one hope that this dual problem might be tractable.

In particular the classes $\mathcal{A}\Av(123,\pi)$ are all easy to describe for $\pi$ a non-monotone permutation of length 3. Specifically:
\begin{itemize}
\item
$\mathcal{A}\Av(123, 132) = \Av(123, 132)$
\item
$\mathcal{A}\Av(123, 213) = \Av(123, 213)$
\item
$\mathcal{A}\Av(123, 231) = 52413[\D, \D, 1, \D, \D]$
\item
$\mathcal{A}\Av(123, 312) = 35241[\D, 1, \D, \D, \D]$
\end{itemize}

The argument in each case is very simple: based on the structure of the permutations, $\tau$, in $\Av(123, \pi)$ given in Table \ref{table-two-three} the linear extensions of $P(\tau)$ can be explicitly listed which, according to Proposition \ref{poset-condition}, provide the elements of $\mathcal{A}\Av(123,\pi)$.

Finally we remark on a more general context for the results in this paper.  We can consider down-sets of permutation pairs other than $\A$ and corresponding analogues of the $\C\longrightarrow \C\A$ operator.  These down-sets $\mathcal{B}$ arise in just the same way as they arise for pattern classes by forbidding one or more pairs to be contained in the pairs of $\mathcal{B}$.  However natural examples are not so readily found.  One such is the set defined by the forbidden pair $(12,21)$ which defines the weak order and some preliminary work on this case may be found in \cite{albert:compositions-of:}.

\bibliographystyle{acm}
\bibliography{refs}

\end{document}